\documentclass[10pt, a4paper, reqno]{amsart}
\numberwithin{equation}{section}
\usepackage{amsmath,amssymb,amsthm}
\usepackage{xcolor,textcomp}
\usepackage{graphicx, epstopdf}
\usepackage{braket,amsfonts}
\usepackage{dsfont}
\usepackage{mathtools}
\usepackage{mwe}
\usepackage[top=3cm,bottom=3cm,left=3cm,right=3cm]{geometry}
\usepackage{hyperref}
\hypersetup{colorlinks=true,linkcolor=red,citecolor=blue,filecolor=magenta,urlcolor=cyan} 
\usepackage[capitalise]{cleveref}
\usepackage{comment}
\usepackage{tikz}
\usepackage{pgfplots}
\usepackage{stmaryrd}
\usepackage[mathscr]{euscript}
\usepackage[sort,nocompress]{cite}
\usepackage{nicefrac}

\newtheorem{theorem}{Theorem}[section]
\newtheorem{lemma}[theorem]{Lemma}
\newtheorem{remark}[theorem]{Remark}

\newtheorem{definition}[theorem]{Definition}

\usepackage{amsopn}
\DeclareMathOperator*{\esssup}{ess\,sup}

\DeclareMathOperator{\Supp}{Supp}

\newcommand\B{\mathscr{B}}

\newcommand\R{\mathbb{R}}

\newcommand\C{\mathcal{C}}

\newcommand\Ss{\mathbb{S}}

\newcommand\veps{\varepsilon}
\newcommand \T{\mathcal T}
\newcommand{\boldvphi}{\boldsymbol{\vphi}} 
\newcommand{\boldphi}{\boldsymbol{\phi}}

\newcommand\y{\mathbf y}

\newcommand\z{\mathbf z}

\newcommand\Z{\mathbf Z}

\newcommand{\disp}{\displaystyle}

\newcommand{\uvect}{\mathbf{u}}
\newcommand{\vbold}{\mathbf{v}}

\newcommand{\Div}{\textnormal{div}}

\newcommand{\x}{\mathbf{x}}

\newcommand{\vphi}{\varphi}
\renewcommand{\rho}{\varrho}

\def\dx{\,\textnormal{d}x}

\def\d{\,\textnormal{d}}

\makeatletter
\@namedef{subjclassname@2020}{%
	\textup{2020} Mathematics Subject Classification}
\makeatother

\title[Compressible fluid on the domains with rough boundaries]{Asymptotic limit of the compressible 
Navier-Stokes system on domains with rough boundaries}

\author[Bhandari, Gahn,  Ne{\v{c}}asov{\'{a}},  Neuss-Radu,  Rodr{\'{\i}}guez-Bellido]
{Kuntal Bhandari$^\dagger$ \and Markus Gahn$^\ddagger$  \and {\v{S}}{\'{a}}rka Ne{\v{c}}asov{\'{a}}$^{\star}$ \and 
Maria Neuss-Radu$^{\S}$  \and 
Mar{\'{\i}}a {\'{A}}ngeles Rodr{\'{\i}}guez-Bellido$^{\P}$
 }
\thanks{$^{\dagger}$Institute of Mathematics,  Czech Academy of Sciences, \v{Z}itn\'a 25, 11567 Praha 1, Czech Republic; \\ \texttt{bhandari@math.cas.cz}}
\thanks{$^\ddagger$Interdisciplinary Center for Scientific Computing, Ruprecht-Karls-Universit\"{a}t Heidelberg
D-69120 Heidelberg, Germany; 
\texttt{markus.gahn@iwr.uni-heidelberg.de}}
\thanks{$^{\star}$Institute of Mathematics,  Czech Academy of Sciences, \v{Z}itn\'a 25, 11567 Praha 1, Czech Republic; \\ 
\texttt{matus@math.cas.cz} (Corresponding author) }
\thanks{$^\S$Department Mathematik, Friedrich-Alexander-Universit{\"{a}}t Erlangen-N{\"{u}}rnberg, D-91058 Erlangen, Germany;  \texttt{maria.neuss-radu@math.fau.de}} 
\thanks{$^\P$Departamento de Ecuaciones Diferenciales y An{\'{a}}lisis Num{\'{e}}rico and IMUS, Facultad de Matem{\'{a}}ticas, Universidad de Sevilla, (Campus de Reina Mercedes), 41012 Seville, Spain;  \texttt{angeles@us.es}}


\date{\today}

\begin{document}

	\begin{abstract}
In this paper, we study the asymptotic behavior of solutions to the compressible Navier-Stokes system considered on
a sequence of spatial domains, whose boundaries exhibit fast oscillations with amplitude and characteristic
wave length proportional to a small parameter. Imposing the full-slip boundary conditions we show
that in the asymptotic limit the fluid sticks completely to the boundary, provided the oscillations are non-degenerate,
meaning not oriented in a single direction.
	\end{abstract}

\maketitle

\section{Introduction}

\subsection{General setting and problem statement}

In the framework of continuum fluid mechanics, the motion of a viscous compressible flow is governed by the  Navier-Stokes system which consists of the continuity and  momentum equations, expressed as 
\begin{align}
	&\partial_t \rho + \Div_x ( \rho \uvect) = 0  ,   \label{continuity-eq} 
	\\
	&\partial_t (\rho \uvect) + \Div_x (\rho [\uvect \otimes \uvect])  + \nabla_x p(\rho)    = \Div_x  \Ss (\nabla_x \uvect)  
	 \label{momentum-eq} ,
\end{align} 
where $\uvect$ is the velocity and $\rho$ is the density of the fluid.  The (barotropic) pressure term is given by  
$$p(\rho)= a \rho^\gamma ,$$ 
for some constant $a>0$, where 
the adiabatic constant $\gamma$  verifies
\begin{align*}
	\gamma > \frac 32 .
\end{align*}
Note that, the case of monoatomic gas, that is when $\gamma=\frac 53$, is included in the above choice. 
 	  
We assume that the fluid is Newtonian, meaning that the viscous stress
tensor $\Ss$ depends linearly on the velocity's gradient. Precisely, we consider  \begin{align}\label{stress_tensor}
	\Ss(\nabla_x \uvect) = \mu \left( \nabla_x \uvect + \nabla^\top_x \uvect -\frac{2}{3} \Div_x \uvect \mathbb I  \right) +\eta \Div_x \uvect \mathbb I,
\end{align}
with shear viscosity coefficient $\mu>0$ and bulk viscosity coefficient $\eta\geq 0$.

A proper choice of boundary conditions plays a crucial role in the problems studied in
continuum fluid dynamics.  In several theoretical studies  and numerical experiments (see for instance, the survey paper by   Priezjev and Troian \cite{priezjev2006influence}), the  
standard well-accepted hypothesis states that a viscous fluid adheres completely to the boundary of the physical domain provided the latter is impermeable. This means that the Eulerian velocity   $\uvect$ 
confined to a bounded domain $\Omega \subset \mathbb R^3$ satisfies 
\begin{align}\label{boundary-imperb}
\uvect \cdot \mathbf n = 0 \ \text{ on } \partial \Omega .
\end{align} 
On the other hand, the complete adherence can be formulated in terms of the no-slip boundary condition
\begin{align}\label{no-slip-condition}
\uvect = 0 \text{ on } \partial \Omega . 
\end{align}

In the present work, we assume that all the quantities are periodic in the plane direction  $(x_1,x_2)$  with periodicity $(1,1)$. Specifically, the system \eqref{continuity-eq}--\eqref{momentum-eq} is considered on a family of bounded domains 
$\{\Omega_\veps\}_{\veps>0}$, given by 
\begin{align}\label{domain}
	\begin{dcases} 
	\Omega_\veps  = \left\{ x = (x_1, x_2, x_3) \mid (x_1,x_2)\in \T^2 , \ 0 < x_3 < 1+ \Phi_\veps(x_1, x_2) \right\}, \\
	\partial \Omega_\veps = \B \cup \Gamma_\veps , \\
	\B = \left\{  (x_1,x_2,x_3) \mid (x_1, x_2) \in \T^2 , \ x_3= 0 \right\}, \\
	\Gamma_\veps = \left\{  (x_1,x_2,x_3) \mid (x_1, x_2) \in \T^2 , \ x_3=1+\Phi_\veps(x_1, x_2) \right\} ,
	\end{dcases} 
\end{align}
where  $\T^2 = ((0,1)|_{\{0,1\}})^2$ stands for the two-dimensional torus, and  $\{\Phi_\veps\}_{\veps>0}$ is  a family of scalar functions with $\Phi_\veps\in \C^{2+\nu}(\T^2)$ for some $\nu>0$  (we refer to Section \ref{Sec-rugosity} for the rigorous construction).



We impose the  no-slip boundary condition for the velocity $\uvect$, i.e.,   
\begin{align}\label{no-slip}
	\uvect|_\B = 0 
\end{align}
on the bottom part $\B$,
and the full-slip condition 
\begin{align}\label{full-slip}
	\uvect\cdot \mathbf n|_{\Gamma_\veps} = 0 , \ \ (\Ss \mathbf n) \times \mathbf n |_{\Gamma_\veps} =0 
\end{align}
on the top part $\Gamma_\veps$.

 Under the assumption that for $\veps \to 0$, it holds
\begin{align}\label{limit-Phi-eps}
\Phi_\veps \to 0 \, \ \text{in} \ \, \C^0(\T^2) ,
\end{align}
%
our objective is to identify the  limit problem for $\veps \to 0$ of the compressible Navier-Stokes system \eqref{continuity-eq}--\eqref{momentum-eq}, where we mainly focus on the situation when, roughly speaking, the top boundary $\Gamma_\veps$ is oscillating with frequency proportional to $\frac{1}{\veps}$ and amplitude $\veps$. 


Mathematically, the problem splits into two parts: finding the limit system as $\veps\to 0$, and identifying the boundary conditions on the target domain 
\begin{align}\label{target_domain}
	\Omega = \{x=(x_1, x_2, x_3)  \mid  (x_1, x_2)\in \T^2  , \, 0<x_3< 1 \} .
\end{align}

\subsection{Bibliographic comments}

To give a rigorous mathematical justification of \eqref{no-slip-condition} based on the concept of rough or rugous boundary, several studies have been pursued over the past centuries. 
Let us start with the early mathematical works by Amirat, Simon and collaborators in \cite{AS96,AS97,Amirat-1} on the Couette flow over a rough plate and applications to {\it drag reduction}. Their approach has similarities with the boundary layers theory from \cite{ALAM99} and \cite{LIO81}
and they concentrate on small Reynolds numbers.
Modeling and computational of flow problems over
rough surfaces was studied  by  Pironneau and collaborators in the papers 
 \cite{MPV98,APV98a,APV98b}. In particular, the
stationary incompressible flow at high Reynolds number {\bf Re}
$\sim \frac{1}{\varepsilon}$ ($\veps>0$ small) over a periodic rough boundary with
the roughness period $\varepsilon$ is considered in the paper \cite{APV98b}.   

\vskip 2pt 

The main point is to indicate that the  boundary of a domain may not be perfectly smooth but contains microscopic asperities of the size significantly smaller than the characteristic length scale of the flow.  In such situations,   the ideal physical domain $\Omega$ is being replaced by a family $\{\Omega_\veps\}_{\veps>0}$ of rough domains, where the parameter $\veps>0$ stands for the amplitude of asperities.  Considering  only the {\em impermeability condition} \eqref{boundary-imperb}  on $\partial \Omega_\veps$,  one can eventually show that the {\em no-slip  boundary condition} must be imposed
for the limit problem when $\Omega_\veps \to \Omega$ (in some sense), provided the distribution of asperities
is uniform, more specifically spatially periodic, and non-degenerate (see \cite[Theorem 1]{Casado-Diaz}). 
%


\vskip 2pt 

 First mathematically rigorous derivation of Navier's slip condition is due to J\"ager and Mikeli\'c in \cite{Mikelic-Jager}. They considered the stationary Navier-Stokes equations for flow governed by a pressure drop and no-slip boundary conditions on a rough part of the boundary. The derivation
is based on the approach from \cite{JM96} and \cite{JM00} in obtaining the {\it interface laws} between porous media and an
unconfined fluid flow.  The auxiliary problems from \cite{JM96}, used to calculate the effective coefficients for the Beavers and Joseph interface law, were slightly modified and used in obtaining the effective Navier's
slip condition. 
The Couette flow around a rough boundary (riblets) was treated in \cite{JM01}. The rigorous derivation of the Navier's slip condition for a general incompressible viscous flow over a rapidly oscillating surface,  was given in \cite{Mikelic_et_al} by Mikeli\'{c}, Ne\v{c}asov\'{a} and Neuss-Radu.
%
All these results were obtained for
{\it periodic} rough boundaries. In fact, realistic natural rough
boundaries are {\it random}  and for study of flows in such geometries we refer to the works of G\'erard-Varet and
collaborators \cite{Basson-Gerard,GV09,GVM10}.



\vspace*{.2cm}

For incompressible Navier-Stokes systems on the rough domains like \eqref{domain},  Bucur et al. \cite{Bucur-Feireisl-Necasova-Wolf} proved that the sequence of velocities $\{\uvect_\veps\}_{\veps}$ is bounded in $L^2(0,T; 
 W^{1,2}(\Omega_\veps; \mathbb R^3))$ and that
 \begin{align} 
 \uvect_\veps \to \uvect \ \text{ in } L^2(0,T; 
 W^{1,2}(\Omega; \mathbb R^3)) \, \text{ as } \, \veps \to 0.
 \end{align} 
In fact, starting with the boundary conditions \eqref{no-slip}--\eqref{full-slip} for the quantities $\uvect_\veps$, the authors in \cite{Bucur-Feireisl-Necasova-Wolf} have shown that the limiting velocity $\uvect$ satisfies the no-slip boundary conditions on both parts of the boundary.

Of course, such a result holds under some certain non-degeneracy condition on 
$\{\Phi_\veps\}_{\veps}$, more specifically, both $\partial_{x_1}\Phi_\veps$ and $\partial_{x_2}\Phi_\veps$ are not identically zero in $\T^2$.  
It was observed in \cite{Bucur-Feireisl-Necasova-Wolf} that such a phenomenon is intimately related to the character of oscillations of the vector fields $\{\nabla_{\overline{x}} \Phi_\veps \}_{\veps>0}$ ($\overline{x}=(x_1,x_2)\in \T^2$) 
associated to the directional fluctuations of the normal vectors to $\Gamma_\veps$. A sufficient condition for \eqref{no-slip}--\eqref{full-slip} to imply \eqref{no-slip-condition}
in the asymptotic limit can be expressed in terms of  the ``rugosity measure", which is simply a Young measure
associated to the family $\{\nabla_{\overline{x}} \Phi_\veps \}_{\veps>0}$. 
More specifically, conditions  \eqref{no-slip}--\eqref{full-slip} give rise
to \eqref{no-slip-condition} in the asymptotic limit provided the support of each
``measure"  contains at least two linearly independent vectors in $\mathbb R^2$ (see \cite[Theorem 4.1]{Bucur-Feireisl-Necasova-Wolf}).

For the case of stationary incompressible Navier-Stokes system,  a similar analysis can be found in  \cite{Bucur-Feireisl-Necasova-ribbed},  when rugosity is degenerate only 
in one direction; in other words, the support of each measure is contained in a
linear subspace of $\mathbb R^2$. 
We also mention  the work \cite{Bucur-Feireisl} by Bucur \& Feireisl where  the authors analyzed the incompressible limit of Navier-Stokes-Fourier system on condition
that the boundary of the spatial domain oscillates with the amplitude and wave length
proportional to the Mach number.   It is worth noting that, in the case of {\em friction-driven boundary conditions} for  incompressible Navier-Stokes systems on the domains with rough boundaries, the authors in  \cite{Friction-Driven-bucur-et-al} studied 
the asymptotic  limit  by means of
$\Gamma$-convergence arguments.

\vspace*{.2cm}

Unlike the other related works, in the present paper, we perform the analysis in the context of compressible Navier-Stokes system.  In particular, we need to use here the concept of Bogovskii operator to obtain the uniform pressure estimate w.r.t. $\veps$, and  the no-slip boundary condition is obtained from the uniform a-priori bound for $\uvect_{\veps}$ and Lemma \ref{Lemma-boundary}. However, of course the technique to obtain the a-priori estimate is different for the compressible case than the incompressible one \cite{Bucur-Feireisl-Necasova-Wolf}.






\section{Geometry of the physical space}\label{Sec-rugosity}

Recall the family of domains  $\{\Omega_\veps\}_{\veps>0}$   introduced in \eqref{domain}, and let us  
  explicitly  consider the family 
$\{\Phi_\veps\}_{\veps>0}$ as  follows:
\begin{align}\label{func-Phi_veps}
\Phi_{\veps} (x_1, x_2) = \veps \, \Phi \bigg(\frac{x_1}{\veps}, \frac{x_2}{\veps}  \bigg)
 \quad \forall (x_1, x_2) \in \T^2 ,
 \end{align} 
where  
$\Phi \in \C^{2+\nu}(\mathbb R^2)$ is a positive function which is periodic w.r.t. the  variables $(\frac{x_1}{\veps}, \frac{x_2}{\veps})$ with period $(1,1)$.
%
		%
 We consider the function  $\Phi$ such that 
\begin{align}\label{func-phi-prop}
\left|\Phi(\y) - \Phi(\z)  \right| \leq L |\y- \z|  \  \text{ for any } \y, \z \in \T^2, 
\end{align}
with constant  $L>0$.  This implies that the family $\{\Phi_\veps\}_{\veps>0}$ is equi-bounded and equi-Lipschitz with the same constant $L$ from \eqref{func-phi-prop}.

 Further, we assume that  $\Phi$ varies in any direction $\y\in \T^2$, at least at one point $\mathbf z \in \T^2$, meaning that 
\begin{align}\label{condition-on-Phi}
    \forall \y \in \T^2 , \  \y \neq 0, \ \text{there exists $\mathbf z \in \T^2$ and $c\in \mathbb R$ such that } \Phi(\mathbf z + c \y) \neq \Phi(\y). 
\end{align}

\begin{remark} 
Note that we have the following equivalent property: For $\boldsymbol{\xi} \in \R^3$ with $\boldsymbol{\xi} \cdot  \mathbf n|_{\color{red}\Gamma_\veps} = 0  $ on $\Gamma_\veps$ it follows $\boldsymbol{\xi} = 0$ or equivalently $\mathrm{span}\{\mathbf n|_{\Gamma_\veps}(\x) \, : \, \x \in \Gamma_\veps \} = \R^3 $.
\end{remark}

Under the above assumption, we recall the following result which has been addressed for instance in \cite[Proposition 5.1]{Bucur-Feireisl}.   
%
\begin{lemma}[Characterization of uniform rugosity effect]\label{Lemma-boundary}
Let $\Omega_\veps \subset \mathbb R^3$ be the family of domains defined through \eqref{domain} where $\Phi_\veps$ is introduced in \eqref{func-Phi_veps}--\eqref{func-phi-prop} and $\Phi$ satisfies \eqref{condition-on-Phi}.

Then there exists $\veps_0>0$, $c_1>0$ independent of $\veps \in (0,\veps_0)$, such that 
\begin{align}\label{boundary-v}
\int_{\{x_3=1\}} |\vbold|^2 \d \sigma \leq \veps c_1  \| \nabla_x \vbold\|^2_{L^2(\Omega_\veps; \mathbb R^{3\times 3})} \quad \text{for all } \veps \in (0,\veps_0) ,
\end{align} 
 for any $\vbold \in W^{1,2}(\Omega_\veps; \mathbb R^3)$ with 
 $$  \vbold \cdot \mathbf n|_{\{x_3 = 1+ \Phi_\veps(x_1, x_2) \}} = 0 .$$
 \end{lemma} 
 

\begin{remark} 
We emphasize that under the conditions of Lemma \ref{Lemma-boundary} one can also obtain the  stronger estimate
\begin{align*}
    \int_{\{x_3=1\}} |\vbold|^2 \d \sigma \leq \veps c_1  \| \nabla_x \vbold\|^2_{L^2(\{1<x_3<\Phi_{\veps} (x_1, x_2)\}; \mathbb R^{3\times 3})}.
\end{align*}
This follows easily from the Poincar\'e inequality in \cite[Lemma 5.1]{fabricius2023homogenization}, the proof of which is obviously valid in our case, and a simple decomposition argument for the rough part $\{1<x_3<\Phi_{\veps} (x_1, x_2)\}$.
\end{remark} 

\begin{remark}
The condition \eqref{condition-on-Phi} is intimately related to the characterization of the vector fields $\{\nabla_{\overline{x}} \Phi_\veps\}_{\veps>0}$. 

In this regard, let us introduce the notion of rugosity. 
A  ``measure of rugosity'', denoted by  $\{\mathcal R_{\y}\}_{\y\in \T^2}$, is  a Young measure associated to the family of gradients $\{\nabla_{\overline{x}} \Phi_\veps  \}_{\veps>0}$. To be more precise,   $\{\mathcal R_{\y}\}_{\y\in \T^2}$ is a family of probability measures on $\R^2$ depending measurably on $\y$  such that 
\begin{align}\label{func-caratheodory-G}  
\int_{\R^2} G(\y, \Z) \d \mathcal R_{\y}(\Z) = \textnormal{weak} \, \lim_{\veps \to 0} G(\y, \nabla_{\overline{x}} \Phi_\veps)   \, \ \text{for a.a. } \, \y \in \T^2, 
\end{align} 
for any Carath\'eodory function $G: \T^2 \times \mathbb R^2\to \mathbb R$ (see Theorem 6.2 by  Pedregal \cite{book-Pedregal}).

Then, the non-degeneracy condition on the family  $\{\Phi_\veps\}_{\veps>0}$, in terms of rugosity measure,  can be interpreted as follows:  
we say that a rugosity measure $\{\mathcal R_\y\}_{\y \in \T^2}$ is non-degenerate if  $\,\Supp [\mathcal R_\y]$ contains two linearly  independent vectors in $\R^2$.  Indeed, it is worth mentioning that  the equivalence  between this  non-degeneracy condition and the condition \eqref{condition-on-Phi} has been proved in \cite[Corollary 4.1]{Bucur-Feireisl-Necasova-Wolf}.   For more details about rugosity measure, we refer \cite{Bucur-Feireisl-Necasova-Wolf}; see also \cite{Bucur-Feireisl-Necasova-ribbed}. 
\end{remark}

\section{Weak solutions in  oscillatory domains}

\begin{definition}\label{Defi-sol-main}
	We say that $(\rho_\veps, \uvect_\veps)$ is a weak solution to the Navier-Stokes system \eqref{continuity-eq}--\eqref{momentum-eq}  on the domain $(0,T)\times \Omega_\veps$, supplemented with the boundary conditions \eqref{no-slip}--\eqref{full-slip},
	  and initial conditions (independent of $\veps)$
	\begin{align}\label{initial}
 \begin{dcases}
	\rho_\veps(0, \cdot) = \rho_{0}  \geq 0   \,  \text{ with } \, \rho_0 \in L^\gamma(\Omega_\veps) , \,
 \text{ and}  \\
 (\rho_\veps \uvect_\veps)(0, \cdot) = \mathbf m_0   \, \text{ with } \, \frac{|\mathbf m_0|^2}{\rho_0} \in L^1(\Omega_\veps) 
 \end{dcases} 
\end{align}
(where $\rho_0$ and $\mathbf m_0$ are taken to be $0$ in $\Omega_\veps \setminus \Omega$), 
if the following holds:
\begin{itemize}

\item $\rho_\veps$ and $\uvect_\veps$ are periodic in the direction $(x_1, x_2)\in \T^2$ with periodicity $(1,1)$. 

	\item $\rho_\veps \geq 0$ a.e. in $ (0,T) \times \Omega_\veps$,  $\rho_\veps \in L^\infty(0,T; L^\gamma(\Omega_\veps))$.  
	
	\item  $\uvect_\veps \in L^2(0,T; W^{1,2}(\Omega_\veps; \mathbb R^3))$,  $\rho_\veps \uvect_\veps  \in L^\infty(0,T; L^{\frac{2\gamma}{\gamma+1}}(\Omega_\veps; \mathbb R^3))$.
	
	\item 
	The continuity equation holds
	\begin{align}\label{weak-for-con}
		- \int_0^T \int_{\Omega_\veps} \left( \rho_\veps  \partial_t \vphi 
		+ \rho_\veps  \uvect_\veps  \cdot \nabla_x \vphi\right) = \int_{\Omega_\veps} \rho_{0} \vphi(0, \cdot) 
	\end{align}
	for any test function $\vphi \in \C^1(\overline{ (0,T)\times \Omega_\veps})$ with $\vphi(T, \cdot) =0$ in $\Omega_\veps$  and periodic in the direction $(x_1, x_2)\in \T^2$ with periodicity $(1,1)$.
	
 Moreover, the equation \eqref{continuity-eq} is satisfied in the sense of renormalized solutions introduced by DiPerna and Lions \cite{Lions-Diperna}: for a.a. $t\in (0,T)$, it holds 
\begin{align}\label{weak-renorm}
&\int_{\R^3} b(\rho_\veps(t))\vphi(t, \cdot)	-	\int_0^t \int_{\R^3}  b(\rho_\veps ) \left( \partial_t \vphi +     \uvect_\veps  \cdot \nabla_x \vphi \right) \notag \\
	&=  \int_{\R^3} b(\rho_{0})  \vphi(0, \cdot) -\int_0^t \int_{\R^3} \left[\rho_\veps b^\prime(\rho_\veps) - b(\rho_\veps) \right] \Div_x \uvect_\veps \vphi 
\end{align}
for any   test function $\vphi\in \mathcal D([0,T) \times \R^3)$ and any   $b \in \C^1([0,\infty))$, where $\rho_\veps, \uvect_\veps$ and $\rho_0$ are extended to be zero outside $\Omega_\veps$.  


\item The momentum equation is satisfied in the following sense
	\begin{align}\label{weak_form_momen}
		-\int_0^T \int_{\Omega_\veps} \left( \rho_\veps \uvect_\veps \cdot \partial_t  \boldvphi   +  \rho_\veps[\uvect_\veps \otimes \uvect_\veps] : \nabla_x  \boldvphi + p(\rho_\veps) \Div_x \boldvphi   \right)     - \int_{\Omega_\veps}   \mathbf m_0\cdot  \boldvphi(0, \cdot) \nonumber \\
		= - \int_0^T \int_{\Omega_\veps} 
		 \Ss(\nabla_x \uvect_\veps) : \nabla_x \boldvphi 	
	\end{align}
for  any test function $\boldvphi \in \C^1( \overline{(0,T) \times \Omega_\veps} ; \mathbb R^3)$, periodic  in $(x_1, x_2)\in \T^2$ with periodicity $(1,1)$,   satisfying 
\begin{align*}
	\boldvphi(T,\cdot)=0 \  \, \text{in} \ \, \Omega_\veps, \ \,  \boldvphi|_{\B} = 0  \ \, \text{and} \, \  	\boldvphi \cdot \mathbf n |_{\Gamma_\veps} = 0 .
\end{align*}

\item The  energy inequality  reads as  
\begin{align} 
\label{energy-ineq}
	\int_{\Omega_\veps}\left(\frac{1}{2} \rho_\veps |\uvect_\veps|^2 + P(\rho_\veps) \right)(\tau, \cdot)    + \int_0^\tau \int_{\Omega_\veps}   \left(\frac{\mu}{2} \Big| \nabla_x \uvect_\veps + \nabla^\top_x \uvect_\veps -\frac{2}{3} \Div_x \uvect_\veps \mathbb{I} \Big|^2 + \eta  |\Div_x \uvect_\veps \mathbb{I}|^2\right) \nonumber \\
	  \leq  	\int_{\Omega_\veps}\left( \frac{|\mathbf m_0|^2}{2\rho_{0}} + P(\rho_{0})  \right) 
	\end{align}
for a.a. $\tau \in [0,T]$, where 
\begin{align}\label{func-P}
	P(\rho_\veps) = \rho_\veps \int_1^{\rho_\veps} \frac{p(z)}{z^2}.
\end{align}
	
	\end{itemize}

\end{definition}

We have the following global (weak) existence result which 
 follows 
from the general theory developed in \cite{Feireisl2001existence}.  
\begin{theorem}\label{Thm-weak-sol} 
Let $\veps>0$ and $\Omega_\veps \subset \mathbb R^3$ be given by \eqref{domain} with the upper part determined by a positive function $\Phi_\veps \in \C^{2+\nu}(\T^2)$ as described in Section \ref{Sec-rugosity}. Then, the Navier-Stokes system \eqref{continuity-eq}--\eqref{momentum-eq} with initial conditions \eqref{initial} and boundary conditions \eqref{no-slip}--\eqref{full-slip} admits a globally defined weak solution in the sense of Definition \ref{Defi-sol-main}.    
\end{theorem}

\begin{proof}
Let us give a short explanation of the proof for Theorem \ref{Thm-weak-sol}.   Based on \cite{Feireisl2001existence} (see also \cite{Feireisl-Novotny-book}), 
it is necessary to consider the continuity equation with a viscous term $\varsigma \nabla_x \rho_\veps$ and introduce artificial pressure term $\delta \rho^{\beta}_\veps$ for $\beta \geq 4$ in the momentum equation, where $\varsigma $ and $\delta$ are small positive parameters.   More precisely, one needs to consider the system 
\begin{align}
	&\partial_t \rho_\veps + \Div_x ( \rho_\veps \uvect_\veps) = \varsigma  \Delta_x \rho_\veps  ,   \label{continuity-eq-epsilon} 
	\\
	&\partial_t (\rho_\veps \uvect_\veps) + \Div_x (\rho_\veps [\uvect_\veps \otimes \uvect_\veps])  + \nabla_x \left( p(\rho_\veps) + \delta \rho^\beta_\veps \right) +\varsigma \nabla_x \rho_\veps \nabla_x \uvect_\veps   = \Div_x  \Ss (\nabla_x \uvect)  
	 \label{momentum-eq-epsilon} ,
\end{align} 
where the equation \eqref{continuity-eq-epsilon} is complemented with the homogeneous Neumann boundary condition
\begin{align*}
\nabla_x \rho_\veps \cdot \mathbf n = 0 \ \text{ on } (0,T) \times \partial \Omega_\veps . 
\end{align*}
The momentum equation is solved via  a Faedo-Galerkin approximation and  then we let $\varsigma \to 0$ to get rid of the artificial viscosity term. In the final step, we need to pass to the limit $\delta \to 0$ to remove the artificial pressure term $\delta \rho^\beta_\veps$. Of course, throughout the process  we need to find  several uniform bounds  w.r.t. $\varsigma$ and $\delta$ as described in \cite{Feireisl2001existence}.
\end{proof}

\section{Weak solutions in the target domain and main result}

Let us define  the notion of weak solution in the target spatial domain $\Omega$ (given by \eqref{target_domain}). 
\begin{definition}\label{Def-limit-sol}
	We say that $(\rho, \uvect)$ is a weak solution to the Navier-Stokes system \eqref{continuity-eq}--\eqref{momentum-eq}  on $(0,T)\times \Omega$, supplemented with the no-slip 
	boundary condition 
	\begin{align}\label{no-slip-2}
	\uvect |_{\partial \Omega} = 0  ,
\end{align}
	 and initial conditions 
\begin{align}\label{initial-2}
 \begin{dcases}
	\rho(0, \cdot) = \rho_{0}     \, \text{ with } \ \rho_0 \in L^\gamma(\Omega) \, 
 \text{ and}  \\
 (\rho \uvect)(0, \cdot) = \mathbf m_0  \, \text{ with } \, \frac{|\mathbf m_0|^2}{\rho_0} \in L^1(\Omega) ,
 \end{dcases} 
\end{align}
if the following holds:
\begin{itemize}

   \item $\rho$ and $\uvect$ are periodic in the direction $(x_1, x_2)\in \T^2$ with periodicity $(1,1)$. 

	\item $\rho \geq 0$ a.e. in $(0,T)\times \Omega$,  $\rho \in L^\infty(0,T; L^\gamma(\Omega))$.  
	
	\item  $\uvect \in L^2(0,T; W^{1,2}(\Omega; \mathbb R^3))$,  $\rho \uvect \in L^\infty(0,T; L^{\frac{2\gamma}{\gamma+1}}(\Omega ; \mathbb R^3))$.
	
	\item 
The continuity equation is satisfied  
 \begin{align}\label{weak-for-con-2}
		- \int_0^T \int_{\Omega} \left( \rho  \partial_t \vphi 
		+ \rho  \uvect_\veps  \cdot \nabla_x \vphi\right) = \int_{\Omega} \rho_{0} \vphi(0, \cdot) 
	\end{align}
	for any test function $\vphi \in \C^1(\overline{ (0,T)\times \Omega})$ with $\vphi(T, \cdot) =0$ in $\Omega$  and periodicity $(1,1)$ in the direction $(x_1, x_2)\in \T^2$.


 Moreover, the continuity equation \eqref{continuity-eq} holds in the sense of renormalized solutions, given by 
	\begin{align}\label{weak-renorm-2}
	&\int_{\R^3} b(\rho(t))\vphi(t, \cdot)	-	\int_0^t \int_{\R^3}  b(\rho ) \left( \partial_t \vphi +     \uvect  \cdot \nabla_x \vphi \right) \notag \\
	&=  \int_{\R^3} b(\rho_{0})  \vphi(0, \cdot) -\int_0^t \int_{\R^3} \left[\rho b^\prime(\rho) - b(\rho) \right] \Div_x \uvect \vphi 
	\end{align}
	for a.a. $t\in (0,T)$,  for any   test function $\vphi\in \mathcal D([0,T) \times \R^3)$ and any   $b \in \C^1([0,\infty))$, where $\rho, \uvect$ and $\rho_0$ are extended to be zero outside $\Omega$.

	\item The momentum equation
 is satisfied in the following sense
		\begin{align}\label{weak_form_momen-2}
			-\int_0^T \int_{\Omega} \left( \rho \uvect \cdot \partial_t  \boldvphi   +  \rho[\uvect \otimes \uvect] : \nabla_x  \boldvphi + p(\rho) \Div_x \boldvphi   \right)     - \int_{\Omega}   \mathbf m_0 \cdot  \boldvphi(0, \cdot) \notag \\
			= - \int_0^T \int_{\Omega} 
			\Ss(\nabla_x \uvect) : \nabla_x \boldvphi 	
		\end{align}
	for  any test function $\boldvphi\in \C^1(\overline{(0,T)\times \Omega}; \mathbb R^3)$, periodic in $(x_1, x_2)\in \T^2$ with periodicity $(1,1)$, satisfying 
	\begin{align*}
		\boldvphi(T,\cdot)=0  \, \ \text{in} \ \, \Omega \ \, \text{and}  \, \ \boldvphi|_{\partial \Omega} = 0 . 
	\end{align*}
	
	\item Lastly, the  energy inequality reads  as 
	\begin{align}  
		\label{energy-ineq-2}
			\int_{\Omega}\left(\frac{1}{2} \rho |\uvect|^2 + P(\rho) \right)(\tau, \cdot)    + \int_0^\tau \int_{\Omega}   \left(\frac{\mu}{2} \Big| \nabla_x \uvect + \nabla^\top_x \uvect -\frac{2}{3} \Div_x \uvect \mathbb{I} \Big|^2 + \eta  |\Div_x \uvect \mathbb{I}|^2\right) \notag  \\
			\leq  	\int_{\Omega}\left( \frac{|\mathbf m_0|^2}{2\rho_{0}} + P(\rho_{0})  \right) 
		\end{align} 
	for a.a. $\tau \in [0,T]$, where
\begin{align}\label{func-P}
	P(\rho) = \rho \int_1^{\rho} \frac{p(z)}{z^2}.
\end{align}
\end{itemize}
\end{definition}


\vspace*{.3cm}


Let us write the main result of the present article.

\begin{theorem}[Main result]\label{Main-Theorem}
    Let $\{\Omega_\veps\}_{\veps>0}$ be a family of domains defined by \eqref{domain} with the family  $\{\Phi_\veps\}_{\veps>0}$  given by \eqref{func-Phi_veps}--\eqref{func-phi-prop} and the function
    $\Phi$ satisfies \eqref{condition-on-Phi}.
     Let $(\rho_\veps, \uvect_\veps)_{\veps>0}$ be a family of weak solutions to the Navier-Stokes system \eqref{continuity-eq}--\eqref{momentum-eq} in $(0,T)\times \Omega_\veps$ as specified in Definition \ref{Defi-sol-main} with boundary conditions \eqref{no-slip}--\eqref{full-slip} 
and initial data $(\rho_0,  \mathbf m_0)$ with $\rho_0 \in L^{\gamma}(\Omega) $ extended by zero to the whole $\R^3$ and $\frac{|\mathbf m_0|^2}{\rho_0} \in L^1(\Omega)$. 
  
  Then, passing  to the limit as $\veps \to 0$,  we have (up to a subsequence)
  \begin{align}
  &	\rho_\veps \to \rho \, \text{  in } \,  \C_{\text{weak}}([0,T]; L^\gamma(\Omega ) ) ,\\
  &  \uvect_\veps \to \uvect \, \text{ weakly in } \, L^2(0,T; W^{1,2}(\Omega ; \mathbb R^3 ) ) , \\
  &  {\rho_\veps}\uvect_\veps \to {\rho}\uvect \, \text{ in } \, \C_{\textnormal{weak}}([0,T]; L^{\frac{2\gamma}{\gamma+1}}(\Omega; \mathbb R^3)) ,
  	\end{align}
  where  $(\rho, \uvect)$ is a weak solution of the Navier–Stokes system \eqref{continuity-eq}--\eqref{momentum-eq}  in $(0,T)\times \Omega$ with $\Omega = \T^2\times (0,1)$, in the
  sense of Definition \ref{Def-limit-sol}, supplemented with the no-slip boundary condition \eqref{no-slip-2} and initial
  data  $(\rho_0, \mathbf m_0)$.
\end{theorem}



 \section{Proof of the main result}

This section is devoted to prove Theorem \ref{Main-Theorem} which is the main result of our paper.

\subsection{Uniform bounds}\label{section-uniform-bound}

 (i)   We  have the following bounds from \eqref{energy-ineq},  
 \begin{align}
 \label{uni-bound-rho}	&\esssup_{\tau \in (0,T)}\| \rho_\veps (\tau, \cdot)\|_{L^\gamma(\Omega_\veps)} \leq C ,\\
 \label{uni-bound-rho-u}	&\esssup_{\tau \in (0,T)}\|\sqrt{\rho_\veps} \uvect_\veps(\tau, \cdot)\|_{L^2(\Omega_\veps ; \mathbb R^3)} \leq C , 
 \end{align}
with some constant $C>0$ that does not depend on $\veps>0$.

\vspace*{.2cm} 

(ii) Let us write the following result  
 due to \cite[Proposition 4.1]{Bucur-Feireisl} by Bucur and Feireisl.
  
 \begin{lemma}[A generalized Korn's inequality]\label{prop-extension-op}	
Let $\Omega_\veps$ be given by \eqref{domain} and  assume that $r$ is a non-negative scalar function on $\Omega_\veps$ such that 
\begin{align}\label{function-r} 
0< m \leq \int_{\Omega_\veps} r \d x , \ \  \ \int_{\Omega_\veps} r^{\nu} \d x \leq M   
\end{align} 
for a certain $\nu > \max\{1, \frac{3q}{4q-3}\}$, $q\in (1, \infty)$.

Then 
\begin{align}\label{Korn-type-modified}
\left\|\vbold \right\|_{W^{1,q}(\Omega_\veps; \R^3)} \leq C(m,M, q) \left(   \left\| \nabla_x \vbold + \nabla_x^\top \vbold - \frac{2}{3} \Div_x \vbold \mathbb I \right\|_{L^q(\Omega_\veps; \R^{3\times 3})} + \int_{\Omega_\veps} r |\vbold| \d x  \right) 
\end{align} 
for any $\vbold \in W^{1,q}(\Omega_\veps; \R^3)$. In particular, the constant $C$ is independent of $\veps $. 
\end{lemma}

Using the  result in Lemma \ref{prop-extension-op}, our goal is show that the sequence $\{\uvect_\veps\}_{\veps>0}$ is uniformly bounded w.r.t. $\veps>0$.
Indeed, since the initial conditions are independent of $\veps$ (according to Definition \ref{Defi-sol-main}), the 
right-hand side of \eqref{energy-ineq} is uniformly bounded w.r.t. $\veps>0$. As a consequence,  we have, from \eqref{energy-ineq}, that
\begin{align}\label{bound-nabla-u}
\int_0^\tau \int_{\Omega_\veps}   \left(\frac{\mu}{2} \Big| \nabla_x \uvect_\veps  + \nabla^\top_x \uvect_\veps -\frac{2}{3} \Div_x \uvect_\veps \mathbb{I} \Big|^2\right) \leq C ,
\end{align}
for a.a. $\tau \in [0,T]$, where the constant $C>0$ is independent of $\veps$. 
Then, using the generalized Korn's  inequality \eqref{Korn-type-modified}, we have (taking $r=\rho_\veps$ in \eqref{function-r})  
\begin{align}\label{uni-bound-u}
	& \|\uvect_\veps \|_{L^2(0,T; W^{1,2}(\Omega_\veps; \mathbb R^3))}  \notag \\
 &\leq 
 C \left(   \left\| \nabla_x \uvect_\veps + \nabla_x^\top \uvect_\veps - \frac{2}{3} \Div_x \mathbb I \right\|_{L^2(\Omega_\veps; \R^{3\times 3})} + \int_{\Omega_\veps} \rho_\veps |\uvect_\veps| \d x  \right) 
 \notag \\
 & 
\leq C \left(   \left\| \nabla_x \uvect_\veps + \nabla_x^\top \uvect_\veps - \frac{2}{3} \Div_x \mathbb I \right\|_{L^2(\Omega_\veps; \R^{3\times 3})} + \left\| \sqrt{\rho_\veps} \uvect_\veps  \right\|^2_{L^2(\Omega_\veps; \R^3)} + \|\rho_\veps\|_{L^1(\Omega_\veps)}  \right) \notag \\
&\leq C ,
\end{align} 
due to \eqref{uni-bound-rho}, \eqref{uni-bound-rho-u} and \eqref{bound-nabla-u}.

\vspace*{.2cm} 

(iii) Now, we compute 
\begin{align*}
\left[\int_{\Omega_\veps} |\rho_\veps \uvect_\veps|^{\frac{2\gamma}{\gamma+1}}\right]^{\frac{\gamma+1}{2\gamma}} \leq 
\left[\left( \int_{\Omega_\veps}\left|\sqrt{\rho_\veps} \uvect_\veps \right|^2  \right)^{\frac{\gamma}{\gamma+1}} \left(\int_{\Omega_\veps} \rho^\gamma_\veps \right)^{\frac{1}{\gamma+1}} \right]^{\frac{\gamma+1}{2\gamma}} ,
\end{align*} 
and thus by virtue of \eqref{uni-bound-rho} and \eqref{uni-bound-rho-u}, we obtain
\begin{align}\label{bound-rho-u-2}
\|\rho_\veps \uvect_\veps \|_{L^\infty(0,T; L^{\frac{2\gamma}{\gamma+1}}(\Omega_\veps; \R^3))  } \leq C . 
\end{align}

\vspace*{.2cm}

(iv)  We further observe that (on condition $\gamma>\frac 32$)
\begin{align*}
\bigg| \int_{\Omega_\veps} \rho_\veps [\uvect_\veps \otimes \uvect_\veps] \cdot  \boldphi               \bigg| \leq \|\rho_\veps \uvect_\veps\|_{L^{\frac{2\gamma}{\gamma+1}}(\Omega_\veps; \R^3) } \|\uvect_\veps\|_{L^6(\Omega_\veps;\R^3)} \|\boldphi\|_{L^{\frac{6\gamma}{2\gamma-3}}(\Omega_\veps; \R^3)},
\end{align*} 
for any $\boldphi \in L^2(0,T; L^{\frac{6\gamma}{2\gamma-3}}(\Omega_\veps; \R^3))$,
and consequently, 
\begin{align}\label{bound-convective-term} 
\|\rho_\veps [\uvect_\veps \otimes \uvect_\veps] \|_{L^2(0,T;L^{\frac{6\gamma}{4\gamma+3}}(\Omega_\veps; \R^3))     } \leq C ,
\end{align} 
thanks to \eqref{uni-bound-u} and \eqref{bound-rho-u-2}.


\vspace*{.2cm}

(v) {\bf Pressure estimate.}  To obtain the  estimate for the pressure term, we use the technique based on the {\em Bogovskii operator}. To begin with, we consider the auxiliary problem: given 
\begin{align}\label{B-1} 
	g \in \mathcal \C^\infty_c(\Omega_\veps), \ \ \ \int_{\Omega_\veps} g \dx = 0 ,
\end{align}
find a vector field $\mathbf v= \B_\veps [g]$ such that 
\begin{align}\label{B-2}
	\mathbf v \in \mathcal \C^\infty_c(\Omega_\veps; \mathbb R^3) , \ \ \  \Div_x \mathbf v = g \ \ \text{in } \Omega_\veps .
\end{align}
The problem \eqref{B-1}--\eqref{B-2} admits many solutions, but we use the construction due to Bogovskii \cite{Bogovskii}. 
In particular, we write the following result.
\begin{lemma}\label{Lemma-Bogovskii}
	For each $\veps>0$, there is a linear solution operator $\B_\veps$ associated to our problem \eqref{B-1}--\eqref{B-2} such that
\begin{align}\label{Bogovskii-esti} 
\|\B_\veps [g] \|_{W^{m+1, q}(\Omega_\veps; \mathbb R^3)} \leq c(m, q) \|g\|_{W^{m,q}(\Omega_\veps)} ,
\end{align}
and moreover, the constant $c(m,q)>0$   is independent of $\veps$. In particular, the norm of $\B_\veps$ is independent of $\veps$. 

If, moreover, $g$ can be written in the form $g=\Div_x \mathbf G$ for a certain $\mathbf G\in L^r(\Omega_\veps; \mathbb R^3)$, $\mathbf G \cdot \mathbf n |_{\partial \Omega_\veps}= 0$, then 
\begin{align}\label{Bogovskii-esti-2}
\|\B_\veps [g] \|_{L^{r}(\Omega_\veps; \mathbb R^3)} \leq c(r) \|\mathbf G\|_{L^{r}(\Omega_\veps;\mathbb R^3)} 
\end{align}
for any $1<r<\infty$, where the constant $c(r)>0$ is independent of $\veps$. 
\end{lemma}

 For the construction of the Bogosvskii operator we refer to \cite[Chapter III.3]{Galdi}. The norm of $\B_\veps$ is independendent of $\veps$, since the rough boundary $\Gamma_\veps$ is given as the graph of equi-Lipschitz familiy $\{\Phi_\veps\}_{\veps>0}$. We refer to \cite[Section 5.1]{bucur2008influence} for more details.


\vspace*{.2cm} 

{\bf Claim:} Now,  our goal is to find the following  estimate for the pressure term:  
\begin{align}\label{pressure est}
\int_0^T\int_{\Omega_\veps}	p(\rho_\veps)  \rho^{\theta}_\veps  \leq C 
\end{align} 
for some $\theta>0$,  where the constant  $C>0$ is independent of $\veps$.



\vspace*{.2cm} 


To prove \eqref{pressure est}, let us  consider the following function 
\begin{align}
\boldvphi = \psi(t) \B_\veps \left[ \rho^\theta_\veps - \frac{1}{|\Omega_\veps|} \int_{\Omega_\veps} \rho^\theta_\veps(t)  \right] , \ \psi \in \C^\infty_c(0,T), \ 0\leq \psi\leq 1 ,
\end{align}
and hereby, we fix  $\theta = \frac23 \gamma -1>0$ (observe that $\theta\leq \gamma$).  
In the sequel, we  denote $$c(\rho_\veps) = \frac{1}{|\Omega_\veps|} \int_{\Omega_\veps} \rho^\theta_\veps(t).$$
We emphasize that the function $ \boldvphi$ is not an admissible test-function in the sense of Definition \ref{Defi-sol-main}, due to the lack of time regularity. However, using a standard approximation scheme first introduced in \cite{Lions-Diperna}, it is possible to make the following argument rigorous. For more details we refer to \cite[Section 7.9.5]{novotny2004introduction}.

Let us formally use  $\boldvphi$ as a test function in \eqref{weak_form_momen}. In what follows, we get 
\begin{align}\label{modi-pres-frm-momen}
&\int_0^T \psi \int_{\Omega_\veps} a\rho^{\gamma+\theta}_\veps \notag  \\
=&  \int_0^T \psi \, c(\rho_\veps) \int_{\Omega_\veps}  a\rho^\gamma_\veps 
- \int_0^T \psi_t \int_{\Omega_\veps} \rho_\veps \uvect_\veps \,  \B_\veps \left[ \rho^\theta_\veps - c(\rho_\veps)  \right]  
 +   
\int_0^T \psi \int_{\Omega_\veps} \rho_\veps \uvect_\veps  \, \B_\veps \left[ \Div_x (\rho^\theta_\veps \uvect_\veps) \right]    \notag \\ 
&+\int_0^T \psi \int_{\Omega_\veps} \rho_\veps \uvect_\veps  \, \B_\veps \left[(\theta-1) \rho^\theta_\veps \Div_x \uvect_\veps - (\theta-1) \frac{1}{|\Omega_\veps|}\int_{\Omega_\veps}\rho^\theta_\veps \Div_x \uvect_\veps \right] \notag \\
& - \int_0^T \psi \int_{\Omega_\veps} \rho_\veps [\uvect_\veps \otimes \uvect_\veps ] : \nabla_x \B_\veps \Big[ \rho^\theta_\veps - c(\rho_\veps) \Big]  +
\int_0^T \psi \int_{\Omega_\veps}
\Ss(\nabla_x \uvect_\veps) : \nabla_x \B_\veps \left[ \rho^\theta_\veps - c(\rho_\veps) \right] \notag \\
&=  \sum_{j=1}^6 I_j .
\end{align} 

\vspace*{.1cm}

\begin{itemize} 
\item By means of the estimate  \eqref{uni-bound-rho} and the fact that $\Omega \subset \Omega_\veps$ for each $\veps>0$, we get 
\begin{align}\label{esti-I-1}
|I_1| &= \bigg|  \int_0^T \psi 
 \, c(\rho_\veps)\int_{\Omega_\veps}  a\rho^\gamma_\veps  \bigg|  \leq  \frac{|a|}{|\Omega_\veps|}\int_0^T \psi \left(\int_{\Omega_\veps} \rho^\theta_\veps\right)  \left(\int_{\Omega_\veps}\rho^\gamma_\veps \right) \notag \\
  & \quad \leq \frac{|a|}{|\Omega|} \|\rho^\theta_\veps\|_{L^1((0,T)\times \Omega_\veps)} \|\rho^\gamma_\veps\|_{L^1((0,T)\times \Omega_\veps)} \|\psi\|_{L^{\infty}(0,T)} \leq C \|\psi\|_{L^{\infty}(0,T)} ,
\end{align} 
where the constant $C>0$ uniform in $\veps>0$, since $\theta\leq \gamma$.

\vspace*{.1cm}

\item
Next, by H\"older inequality with $\frac12 + \frac{1}{2\gamma} + \frac{\gamma - 1}{2\gamma} = 1$ we get 
\begin{align}
|I_2|  & = \bigg|\int_0^T \psi_t \int_{\Omega_\veps} \rho_\veps \uvect_\veps \,  \B_\veps \left[ \rho^\theta_\veps - c(\rho_\veps)  \right] \bigg| \notag \\
&\leq \int_0^T |\psi_t| \|\sqrt{\rho_\veps} \|_{L^2(\Omega_\veps)} \| \sqrt{\rho_\veps} \uvect_\veps \|_{L^2(\Omega_\veps; \R^3)} \| \B\left[ \rho^\theta_\veps - c(\rho_\veps) \right] \|_{L^{\frac{2\gamma}{\gamma - 1}}(\Omega_\veps; \R^3)}  \notag \\
& \leq C \int_0^T |\psi_t| \| \B\left[ \rho^\theta_\veps - c(\rho_\veps) \right] \|_{W^{1,\frac{6\gamma}{5\gamma - 3}}(\Omega_\veps; \mathbb R^3)} \notag \\
& \leq C \int_0^T |\psi_t| \| \rho^\theta_\veps  \|_{L^{p}(\Omega_\veps)} \leq C\|\psi_t\|_{L^1(0,T)} , 
\end{align} 
 where we have used \eqref{Bogovskii-esti}, the embedding $W^{1,\frac{6\gamma}{5\gamma -3}}(\Omega_\veps)\hookrightarrow L^{\frac{2\gamma}{\gamma -1}}(\Omega_\veps)$, the  fact $\theta \frac{6\gamma}{5\gamma - 3} \le \gamma $, and the estimates \eqref{uni-bound-rho} and  \eqref{uni-bound-rho-u}.

\vspace*{.1cm}

\item By virtue of \eqref{Bogovskii-esti-2},  \eqref{uni-bound-rho} and \eqref{uni-bound-u}, we get 
\begin{align}
|I_3| &= \bigg| \int_0^T \psi \int_{\Omega_\veps} \rho_\veps \uvect_\veps \B_\veps \left[ \Div_x (\rho^\theta_\veps \uvect_\veps)  \right]     \bigg|
\notag \\
&\leq \int_0^T |\psi| \| \rho_\veps\|_{L^\gamma(\Omega_\veps)} \|\uvect_\veps\|_{L^6(\Omega_\veps;\R^3)}\|\B_\veps \left[\Div_x (\rho^\theta_\veps \uvect_\veps)  \right] \|_{L^{\frac{6\gamma}{5\gamma -6}}(\Omega_\veps; \R^3)}  \notag \\
& \leq \int_0^T |\psi|  \| \rho_\veps\|_{L^\gamma(\Omega_\veps)} \|\uvect_\veps\|_{L^6(\Omega_\veps;\R^3)}\|\rho^\theta_\veps \uvect_\veps  \|_{L^{\frac{6\gamma}{5\gamma -6}}(\Omega_\veps; \R^3)}  \notag \\
&\leq 
\|\psi\|_{L^\infty(0,T)} \|\rho_\veps\|_{L^\infty(0,T; L^\gamma(\Omega_\veps) )} \int_0^T\|\uvect_\veps \|_{L^6(\Omega_\veps; \R^3)} \|\rho^\theta_\veps\|_{L^{\frac{3\gamma}{2\gamma-3}}(\Omega_\veps)} \|\uvect_\veps\|_{L^6(\Omega_\veps; \mathbb R^3)} \notag \\  
&\leq \|\psi\|_{L^\infty(0,T)} \|\rho_\veps\|_{L^\infty(0,T; L^\gamma(\Omega_\veps) )} \int_0^T\|\uvect_\veps \|^2_{L^6(\Omega_\veps; \R^3)} \|\rho^\theta_\veps\|_{L^{\frac{3\gamma}{2\gamma-3}}(\Omega_\veps)}  \notag \\  
& \leq C\|\psi\|_{L^{\infty}(0,T)},
 \end{align}
since we have $\theta \frac{3\gamma}{2\gamma - 3} = \gamma$.

\vspace*{.1cm}

 \item   Let us look into the  term $I_4$, we see
 \begin{align} 
|I_4| & = \bigg| \int_0^T \psi \int_{\Omega_\veps}          \rho_\veps \uvect_\veps  \B_\veps \Big[(\theta-1) \rho^\theta_\veps \Div_x \uvect_\veps - (\theta-1) \frac{1}{|\Omega_\veps|}\int_{\Omega_\veps}\rho^\theta_\veps \Div_x \uvect_\veps \Big]    \bigg| \notag \\
& \leq \int_0^T |\psi|  \|\rho_\veps\|_{L^\gamma(\Omega_\veps)} \|\uvect_\veps \|_{L^6(\Omega_\veps; \R^3)}  \notag \\
& \qquad \times \Big\|  \B_\veps \Big[(\theta-1) \rho^\theta_\veps \Div_x \uvect_\veps - (\theta-1) \frac{1}{|\Omega_\veps|}\int_{\Omega_\veps}\rho^\theta_\veps \Div_x \uvect_\veps \Big]            \Big\|_{L^{\frac{6\gamma}{5\gamma-6}}(\Omega_\veps; \R^3)} \notag \\
& \leq \int_0^T |\psi| \|\rho_\veps\|_{L^\gamma(\Omega_\veps)} \|\uvect_\veps \|_{L^6(\Omega_\veps; \R^3)}\notag \\
& \qquad \times \Big\|  \B_\veps \Big[(\theta-1) \rho^\theta_\veps \Div_x \uvect_\veps - (\theta-1) \frac{1}{|\Omega_\veps|}\int_{\Omega_\veps}\rho^\theta_\veps \Div_x \uvect_\veps \Big]            \Big\|_{W^{1,\frac{6\gamma}{7\gamma-6}}(\Omega_\veps; \R^3)} \notag \\
&\leq \int_0^T |\psi| \|\rho_\veps\|_{L^\gamma(\Omega_\veps)} \|\uvect_\veps \|_{L^6(\Omega_\veps; \R^3)} \| (\theta-1) \rho^\theta_\veps \Div_x \uvect_\veps  \|_{L^{\frac{6\gamma}{7\gamma-6}}(\Omega_\veps; \R^3) } \notag \\
& \leq C \|\psi\|_{L^{\infty}(0,T)}\int_0^T \|\rho_\veps\|_{L^\gamma(\Omega_\veps)} \|\uvect_\veps\|_{L^6(\Omega_\veps; \R^3)} \|\Div_x \uvect_\veps \|_{L^2(\Omega_\veps; \R^3)} \|\rho^\theta_\veps\|_{L^{\frac{3\gamma}{2\gamma-3}}(\Omega_\veps) }  \notag \\
& \leq C \|\psi\|_{L^{\infty}(0,T)} \|\rho_{\veps}\|_{L^\infty(0,T; L^\gamma(\Omega_\veps))} \|\rho^\theta_{\veps}\|_{L^\infty(0,T; L^{\frac{3\gamma}{2\gamma-3}}(\Omega_\veps))} \int_0^T \|\uvect_\veps\|^2_{W^{1,2}(\Omega_\veps; \R^3)} \notag \\
&\leq C\|\psi\|_{L^{\infty}(0,T)},
 \end{align} 
where we have used the continuous embedding $W^{1, \frac{6\gamma}{7\gamma-6}}(\Omega_\veps ; \R^3) \hookrightarrow L^{\frac{6\gamma}{5\gamma-6}}(\Omega_\veps,\R^3)$ and the fact that $\theta = \frac{2}{3}\gamma -1$.


\vspace*{.1cm}

 \item The term $I_5$ can be estimated as 
 \begin{align}
|I_5| &= \bigg|   \int_0^T \psi \int_{\Omega_\veps} \rho_\veps [\uvect_\veps \otimes \uvect_\veps ] : \nabla_x \B_\veps \left[ \rho^\theta_\veps - c(\rho_\veps) \right]   \bigg| 
\notag \\
& \leq \int_0^T |\psi| \left\|\rho_\veps |\uvect_\veps|^2\right\|_{ L^{\frac{3\gamma}{\gamma+3}}(\Omega_\veps; \R^{3}) } \| \nabla_x \B_\veps [\rho^\theta_\veps - c(\rho_\veps)] \|_{ L^{\frac{3\gamma}{2\gamma-3}}(\Omega_\veps; \R^3) }   \notag \\
& \leq \int_0^T |\psi| \|\rho_\veps\|_{L^\gamma(\Omega_\veps)} \|\uvect_\veps\|^2_{L^6(\Omega_\veps; \R^3)}  \|\B_\veps [\rho^\theta_\veps - c(\rho_\veps)] \|_{W^{1,\frac{3\gamma}{2\gamma-3}}(\Omega_\veps; \R^3) } \notag \\
& \leq \int_0^T |\psi| \|\rho_\veps\|_{L^\gamma(\Omega_\veps)} \|\uvect_\veps\|^2_{L^6(\Omega_\veps; \R^3)}\|\rho^\theta_\veps\|_{L^{\frac{3\gamma}{2\gamma-3}}(\Omega_\veps)} \leq C\|\psi\|_{L^{\infty}(0,T)} ,
 \end{align} 
by using H\"older inequality, and the estimates \eqref{uni-bound-rho}, \eqref{bound-convective-term} and \eqref{Bogovskii-esti}, where we also used the fact that $\theta = \frac{2}{3}\gamma-1$. 


\vspace*{.1cm}

\item Finally, we compute the following:
\begin{align}\label{esti-I-5}
|I_6| &= \bigg|  \int_0^T \psi \int_{\Omega_\veps} \Ss(\nabla_x \uvect_\veps) : \nabla_x \B_\veps [\rho^\theta_\veps - c(\rho_\veps)]  \bigg| \notag \\
&\leq \int_0^T |\psi| \|\Ss(\nabla_x\uvect_\veps)\|_{L^{2}(\Omega_\veps; \R^3)} \|\nabla_x\B_\veps [\rho^\theta_\veps-c(\rho_\veps)] \|_{L^{2}(\Omega_\veps; \R^3)} \notag \\
&\leq \int_0^T |\psi| \|\uvect_\veps\|_{W^{1,2}(\Omega_\veps; \R^3)} \|\B_\veps [\rho^\theta_\veps-c(\rho_\veps)] \|_{W^{1,2}(\Omega_\veps; \R^3)} \notag \\
& \leq \|\psi\|_{L^\infty(0,T)} \|\uvect_\veps\|_{L^2(0,T; W^{1,2}(\Omega_\veps; \R^3) } \|\rho^\theta_\veps \|_{L^\infty(0,T; L^2(\Omega_\veps))} \leq C\|\psi\|_{L^{\infty}(0,T)}.
\end{align} 
\end{itemize}

\vspace*{.2cm}

Summing up the estimates \eqref{esti-I-1}--\eqref{esti-I-5} in \eqref{modi-pres-frm-momen}, we obtain the required bound \eqref{pressure est} by choosing $\psi = \psi_m$ with $\psi_m \in C_c^{\infty}(0,T)$ with $\psi_m = 1$ in $\left(\frac{1}{m} , T - \frac{1}{m}\right)$ and $|\psi_t|\le 2m$, and consider $m\to \infty$.


\begin{remark}\label{Remark-uniform-bound}
 Note that, from the definition of $\Omega_\veps$ and $\Omega$, we have $\Omega\subset \Omega_\veps$ for any $\veps >0$. 
 As a consequence, all the estimates \eqref{uni-bound-rho}--\eqref{uni-bound-rho-u}, \eqref{uni-bound-u}--\eqref{bound-convective-term} and \eqref{pressure est} are uniform w.r.t. $\veps$ when the space domain is replaced by $\Omega$ instead of $\Omega_\veps$. 
\end{remark}

\subsection{Passing to the limit}\label{Sec-pass-limit}  This subsection is devoted to pass to the limit for $\veps\to 0$.

\vspace*{.2cm} 

(i) From \eqref{uni-bound-u} and Remark \ref{Remark-uniform-bound},    we have 
\begin{align}\label{weak-limit-u}
\uvect_\veps \to \uvect \ \text{ weakly in } \, L^2(0,T; W^{1,2}(\Omega;\mathbb R^3) ) .
\end{align}

 \vspace*{.2cm} 
 (ii)
Then from \eqref{uni-bound-rho} and \eqref{uni-bound-rho-u} (together with Remark \ref{Remark-uniform-bound}), we respectively have   
 \begin{align}\label{limit-L-inf-rho}
 &\rho_\veps \to \rho \ \text{ weakly$^*$ in } \, L^\infty(0,T; L^\gamma(\Omega) ), \\
 &\sqrt{\rho_\veps}   \uvect_\veps \to  \sqrt{\rho} \uvect  \ \text{ weakly$^*$ in } \, L^\infty(0,T ; L^2(\Omega;\mathbb R^3) )   \label{limit-L-inf-rho-sqrt-u} . 
 \end{align}
 Moreover, the limit \eqref{limit-L-inf-rho} can be improved to 
\begin{align}\label{strong-limit-rho}
	\rho_\veps \to \rho \ \text{ in } \, \C_{\text{weak}}([0,T]; L^\gamma(\Omega) ) .
\end{align}  
Indeed, from the continuity equation \eqref{continuity-eq} and \eqref{bound-rho-u-2}, one can check that the functions 
$\disp t \mapsto \left[\int_{\Omega} \rho_\veps \phi\right](t)$ for all $\phi \in \C^\infty_c(\Omega)$, form a bounded and equi-continuous sequence in $\C^0([0,T])$.   Consequently, by standard Arzel\`a-Ascoli theorem, we have 
\begin{align*}
\int_\Omega \rho_\veps \phi 
\to \int_\Omega  \rho \phi \ \  \text{in } \C^0([0,T]) \ \text{for any } \phi \in \C^\infty_c(\Omega)   .
\end{align*} 
Since $\rho_\veps$ satisfies 
the bound \eqref{uni-bound-rho}, the above convergence can be extended  for each $\phi \in L^{\gamma^\prime}(\Omega)$ via density argument and thus, the limit \eqref{strong-limit-rho} holds true.

\vspace*{.2cm} 
(iii)   
 Using the fact in Remark \ref{Remark-uniform-bound}, the bound \eqref{bound-rho-u-2}, and the limits  \eqref{weak-limit-u} and \eqref{strong-limit-rho}, we have
 %
\begin{align}\label{weak-limit-rho-u} 
	\rho_\veps \uvect_\veps \to \rho \uvect \, \text{ weakly$^*$ in } L^\infty(0,T; L^{\frac{2\gamma}{\gamma+1}}(\Omega; \mathbb R^3)) .
\end{align}



\vspace*{.2cm}

(iv)  On the other hand, thanks to the bound \eqref{bound-convective-term}, the convective term satisfies 
\begin{align}
\rho_\veps [\uvect_\veps \otimes \uvect_\veps] \to \overline{\rho [\uvect \otimes \uvect]} \ \text{ weakly in } \, L^2(0,T;   L^{\frac{6\gamma}{4\gamma +3}}(\Omega; \mathbb R^{3\times 3})   )  ,
\end{align} 
where ``bar" denotes the weak limit of the associated composed function. 
 Now, we have to show that 
\begin{align}\label{pointwise-convective} 
\overline{\rho [\uvect \otimes \uvect]} = {\rho [\uvect \otimes \uvect]} \ \text{ a.e. in } (0,T)\times \Omega. 
\end{align} 
Indeed, from the momentum equation \eqref{momentum-eq} and the uniform bounds established in Section \ref{section-uniform-bound}, one can find  that 
the functions 
\begin{align*}
\left\{t \mapsto \int_{\Omega} \rho_\veps \uvect_\veps \cdot \boldphi    \right\} \text{ are equi-continuous and bounded in } \C^0([0,T]) \text{ for any } \boldphi \in \C_c^\infty(\Omega; \R^3).
\end{align*} 
Consequently, by Arzel\`{a}-Ascoli theorem, we deduce  that 
\begin{align*}
\rho_\veps \uvect_\veps \to \rho \uvect \ \text{ in } \, \C_{\textnormal{weak}}([0,T]; L^{\frac{2\gamma}{\gamma+1}}(\Omega; \mathbb R^3)) . 
\end{align*} 
On the other hand, since $L^{\frac{2\gamma}{\gamma+1}}(\Omega)$ is compactly embedded into $W^{-1,2}(\Omega)$, we infer that 
\begin{align*}  
\rho_\veps \uvect_{\veps} \to \rho \uvect \ \text{ strongly in } \, \C_{\textnormal{weak}}([0,T]; W^{-1,2}(\Omega; \R^3) ) .
\end{align*} 
The above limit, together with the weak convergence of the velocities $\{\uvect_\veps\}_{\veps>0}$ in the space $L^2(0,T; W^{1,2}(\Omega; \R^3))$ (see \eqref{weak-limit-u}) give rise to \eqref{pointwise-convective}.



\vspace*{.2cm}

(v) {\bf Pointwise convergence of the density.} 
Next, we show pointwise convergence of the sequence $\{\rho_\veps\}_{\veps>0}$ in $(0,T)\times \Omega$. 

 We  denote 
 $T_k(\rho_\veps)=\min\{\rho_\veps, k\}$.   
Similar to the analysis in \cite{Feireisl-Indiana-2004}, one can find the {\em effective viscous pressure identity}:
\begin{align}\label{effective-idendity}
\overline{p(\rho_\veps) T_k(\rho_\veps)} - \overline{p(\rho)} \ \overline{T_k(\rho)}  = \frac{4}{3} \mu \left( \overline{T_k(\rho) \Div_x \uvect} - \overline{T_{k}(\rho)} \, \Div_x \uvect \right),
\end{align} 
which holds in the domain $(0,T)\times \Omega$.

Now, following \cite{Feireisl-Indiana-2004} (see also \cite[Chapter 6]{Eduard-book-Oxford}), we introduce the
 {\em oscillations defect measure} 
 \begin{align*}
\textbf{osc}_q [\rho_\veps \to \rho]((0,T)\times\Omega)) = \sup_{k\geq 0} \left( \limsup_{\veps \to 0} \int_0^T \int_\Omega \left|T_k(\rho_\veps) - T_k(\rho) \right|^q \right) ,
 \end{align*} 
 and then use \eqref{effective-idendity} (as well as \eqref{pressure est}) to conclude that 
 \begin{align*}
\textbf{osc}_{\gamma+1}[\rho_\veps \to \rho]((0,T)\times\Omega))  \leq C, 
 \end{align*} 
where the constant is independent of $\veps$. This implies the desired conclusion 
\begin{align}\label{a.e.limit-rho}
\rho_\veps \to \rho \, \text{ a.e. in } (0,T) \times \Omega  
\end{align} 
by virtue of the procedure developed in \cite[Chapter 6]{Eduard-book-Oxford}.

\vspace*{.2cm}

(vi) {\bf Limit of the pressure terms.} Finally, from the estimate \eqref{pressure est}, the pressure terms satisfy 
\begin{align}
p(\rho_\veps)  \to \overline{p(\rho)}  \ \text{ weakly in } \, L^1((0,T)\times \Omega).
\end{align}
But due to a.e. convergence of the density \eqref{a.e.limit-rho},  it follows that 
\begin{align*}
\overline{p(\rho)}  = p(\rho)  .
\end{align*}

\subsection{Recovering the no-slip boundary condition}  Let us discuss about the boundary condition satisfied by the limiting velocity $\uvect$. 

\begin{itemize} 
    \item[--]  Recall that $\uvect_\veps$ satisfies 
$\uvect_\veps |_{\{x_3 = 0\}}=0$. Thus, from the weak limit \eqref{weak-limit-u}, one has  (since $\uvect \in L^2(0,T; W^{1,2}(\Omega; \mathbb R^3))$)
\begin{align}\label{bdry-1}
\uvect(t, \cdot)|_{\{x_3=0\}} = 0  \, \text{ for a.a. } t \in (0,T).  
\end{align}

\item[--] Next, our goal is to show that on the part $\{x_3=1\}$ of the boundary of the target domain $\Omega=\T^2\times (0,1)$, the limiting velocity $\uvect$ satisfies the no-slip condition. 

Indeed, since 
$\uvect_\veps\in L^2(0,T; W^{1,2}(\Omega_\veps; \R^3))$ and 
\begin{align*}
\uvect_\veps (t, \cdot)|_{\{x_3=0\}} = 0 \text{ and  } \uvect_\veps(t, \cdot) \cdot \mathbf n |_{\{x_3=1+\Phi_\veps(x_1,x_2)\}} = 0 \, \text{ for a.a. } t \in (0,T),   
\end{align*} 
by virtue of  Lemma \ref{Lemma-boundary}, there exists some $\veps_0>0$ such that one has 
\begin{align*}
\int_{\{x_3=1\}} |\uvect_\veps(t, \cdot)|^2 d \sigma \leq \veps c_1 \|\nabla_x \uvect_\veps(t, \cdot) \|^2_{L^2(\Omega_\veps; \R^3)} \quad \forall \veps \in (0,\veps_0) \text{ for a.a. } t \in (0,T).
\end{align*} 
Integrating w.r.t. $t \in (0,T)$, we have 
\begin{align*}
\int_0^T \int_{\{x_3=1\}} |\uvect_{\veps}(t, \cdot)|^2 d \sigma \leq \veps c_1 \| \uvect_\veps \|^2_{L^2(0,T;  W^{1,2}(\Omega_\veps; \R^3) )} \quad \forall \veps \in (0,\veps_0).
\end{align*}
Hence, as $\veps \to 0$, from the above inequality as  well as  the weak convergence of $\{\uvect_\veps\}_{\veps}$ to $\uvect$ in $L^2(0,T; W^{1,2}(\Omega; \R^3))$, we deduce that  
\begin{align}\label{bdry-2}
\uvect(t, \cdot)|_{\{x_3=1\}} = 0  \, \text{ for a.a. } t \in (0,T) .  
\end{align} 
\end{itemize}


\subsection{Resultant weak formulations} 
\label{subsection:result_weak_formulation}
Thanks to  the limiting behaviors analyzed in Section \ref{Sec-pass-limit}, and the boundary conditions recovered in \eqref{bdry-1}--\eqref{bdry-2}, one can obtain  the weak formulations of the concerned Navier-Stokes system in the target domain $(0,T)\times \Omega$,  as  given in Definition \ref{Def-limit-sol}.

\section*{Acknowledgments}
K. Bhandari and \v{S}. Ne{\v{c}}asov{\'{a}} received funding from   the  Praemium Academiae of \v{S}\'arka  Ne{\v{c}}asov{\'{a}}.  \v{S}. Ne\v{c}asov\'a  was also supported by the Czech Science Foundation (GA\v CR) through project  GA22-01591S. The work of  M. A. Rodr\'{\i}guez-Bellido  was partially supported by project US-1381261 (US/ JUNTA/ FEDER, UE) with the participation of FEDER. 
Finally, the Institute of Mathematics, CAS is supported by RVO:67985840.

\bibliographystyle{siam}
\bibliography{ref_rough}
\end{document}